\documentclass{amsart}
\usepackage{amsfonts,amssymb,amsmath,amsthm}
\usepackage{url}
\usepackage{enumerate}
\usepackage{graphicx}

\newtheorem{thrm}{Theorem}[section]

\theoremstyle{definition}

\newtheorem{remark}[thrm]{Remark}
\numberwithin{equation}{section}

\author{Victor Alexandrov}
\address{
Sobolev Institute of Mathematics\\
Koptyug ave., 4\\
Novosibirsk, 630090, Russia \&
Department of Physics\\
Novosibirsk State University\\
Pirogov str., 2\\
Novosibirsk, 630090, Russia}
\email{alex@math.nsc.ru}
\thanks{The author is supported in part by the 
Russian Foundation for Basic Research 
(project 10--01--91000--anf) 
and the State Maintenance Program
for Young Russian Scientists and 
the Leading Scientific Schools of the
Russian Federation (grant NSh--921.2012.1)}

\begin{document}

\title[Continuous deformations of polyhedra]{Continuous deformations of
polyhedra that do not alter the dihedral angles}

\begin{abstract}
We prove that, both in the hyperbolic and spherical 3-spaces,
there exist nonconvex compact boundary-free polyhedral surfaces 
without selfintersections which admit nontrivial continuous 
deformations preserving all dihedral angles and study 
properties of such polyhedral surfaces.
In particular, we prove that the volume of the domain, 
bounded by such a polyhedral surface, is necessarily 
constant during such a deformation while, for some 
families of polyhedral surfaces, the surface area, 
the total mean curvature, and the Gauss curvature 
of some vertices are nonconstant during deformations
that preserve the dihedral angles.
Moreover, we prove that, in the both spaces, 
there exist tilings that possess nontrivial
deformations preserving the dihedral angles of every tile
in the course of deformation.
\end{abstract}
\maketitle

\section{Introduction} 
We study polyhedra (more precisely, boundary-free 
compact polyhedral surfaces) the spatial shape of which 
can be changed continuously in such a way that all dihedral 
angles  remain constant. 

These polyhedra may be considered as a natural
`dual object' for the flexible polyhedra.
The latter are defined as polyhedra whose spatial shape can 
be changed continuously due to changes of their dihedral 
angles only,  i.\,e., in such a way that every face remains 
congruent to itself during the flex.
Since 1897, it was shown that  flexible polyhedra do exist 
and have numerous nontrivial properties. 
Many authors contributed to the theory of 
flexible polyhedra, first of all we should mention
R. Bricard, R. Connelly, I.Kh. Sabitov, 
H. Stachel, and A.A. Gaifullin.
For more details, the reader is referred to
the survey article \cite{Sa11} and references given there.

In 1996, M.Eh. Kapovich brought our attention to 
the fact that polyhedra, admitting nontrivial deformations
that keep all dihedral angles fixed, may be of some interest 
in the theory of hyperbolic manifolds, where Andreev's 
theorem \cite{An70} plays an important role. The latter reads 
that, under the restriction that the dihedral angles must be 
nonobtuse, a compact convex hyperbolic polyhedron is uniquely 
determined by its dihedral angles.

The case of the Euclidean 3-space is somewhat special and
we do not study it here.
The reader may consult \cite{MM11} and references given there
to be acquainted with the progress in solving old conjectures 
about unique determination of Euclidean polytopes by their 
dihedral angles that may be dated back to J.J. Stoker's paper~\cite{St68}.

Coming back to continuous deformations of hyperbolic or spherical
polyhedra which leave the dihedral angles fixed, we can 
immediately propose the following example. 
Consider the boundary $P$ of the union of a convex polytope $Q$
and a small tetrahedron $T$ (the both are treated as solid 
bodies for a moment) located so that 
$(i)$ a face $\tau$ of $T$ lies inside a face of $Q$ and 
$(ii)$ $T$ and $Q$ lie on the different sides of the plane 
containing $\tau$.

Obviously, the nonconvex  compact polyhedron $P$ 
has no selfintersections and admits nontrivial 
(i.e., not generated by a rigid motion of the whole space) 
continuous deformations preserving all dihedral angles. 
In order to construct such a deformation 
we can keep $Q$ fixed and continuously move (e.g., rotate) 
$T$ in such a way that the conditions $(i)$ and $(ii)$ are 
satisfied.\footnote{If the reader prefers to deal 
with polyhedra with simply connected faces only, 
he can triangulate the face $P\cap\tau$ 
without adding new vertices. In this case, the movement
of $T$ should be small enough so that
no triangle of the triangulation becomes degenerate
during the deformation.}
In this example, many quantities associated with $P$
remain constant. To name a few, we can mention
\begin{enumerate}[(i)]
\item  the volume;
\item  the surface area;
\item  the Gauss curvature of every vertex (i.e., 
the difference between $2\pi$ and the sum of all plane
angles of $P$ incident to this vertex);
\item  the total mean curvature of $P$ (i.e., the sum
\begin{equation}\label{eq1}
\frac12\sum\limits_{\ell} \bigl(\pi-\alpha(\ell)\bigr)|\ell|
\end{equation}
calculated over all edges $\ell$ of $P$, where $\alpha(\ell)$
stands for the dihedral angle of $P$ attached to $\ell$ and
$|\ell|$ for the length of $\ell$);
\item  every separate summand 
$\bigl(\pi-\alpha(\ell)\bigr)|\ell|$
in~\eqref{eq1}.
\end{enumerate}

In the hyperbolic  
and spherical 3-spaces, 
we study whether the above example provides us
with  the only possibility
to construct nonconvex compact polyhedra
that admit nontrivial continuous deformations 
preserving all dihedral angles
and prove that the answer is negative.

We study also what quantities 
associated with nonconvex compact polyhedra
necessarily remain constant 
during such deformations 
and show that the volume of the domain bounded by the
polyhedral surface is
necessarily constant, while the surface area, 
the total mean curvature, and
the Gauss curvature of a vertex
may be nonconstant.

At last, we 
prove that there exist tilings 
that possess nontrivial deformations 
leaving the dihedral angles of every tile fixed
in the course of deformation.
Here we use a construction originally proposed (but never 
published) in the 1980th by A.V. Kuz'minykh at the 
geometry seminar of A.D. Alexandrov in Novosibirsk, Russia,
which was proposed for the study of a similar problem for
flexible polyhedra. 

\section{Polyhedra in the hyperbolic  3-space} 

\begin{thrm}\label{t1}
In the hyperbolic 3-space $\mathbb{H}^3$,
there exists a nonconvex sphere-homeo\-morphic 
polyhedron $P$ with the following properties:
\begin{enumerate}[(i)]
\item $P$ has no selfintersections;
\item $P$ admits a continuous family of 
nontrivial deformations which leave the dihedral 
angles fixed;
\item the surface area of $P$ is nonconstant
during the deformation;
\item the total mean curvature  of $P$ is nonconstant
during the deformation;
\item the Gauss curvature of some vertex of $P$
is nonconstant.
\end{enumerate}
\end{thrm}

\begin{proof}
We divide the proof in a few steps.

\textit{Step I}: Let $S$ be a unit 2-sphere in $\mathbb{H}^3$,
$C\subset S$ be a circle on the 2-sphere, 
and $a\in S$ be a point lying
outside the convex disk on $S$ bounded by $C$.
Draw two geodesic lines $L^1$ and $L^2$ on $S$ 
that pass through $a$ and are tangent to $C$ (see Figure~\ref{fig1}). 
Now rotate the set $L^1\cup L^2$ to an arbitrary 
angle around the center of $C$ and denote the image of
$L^j$ by $\overline{L^j}$, $j=1,2$. 
Let $c$ be the image of the point $a$ under the above rotation,
$b=L^1\cap \overline{L^1}$, and $d=L^2\cap \overline{L^2}$. 
The resulting  configuration is schematically shown
on Figure~\ref{fig1}.
\begin{figure}\label{fig1}
\begin{center}
\includegraphics[width=0.5\textwidth]{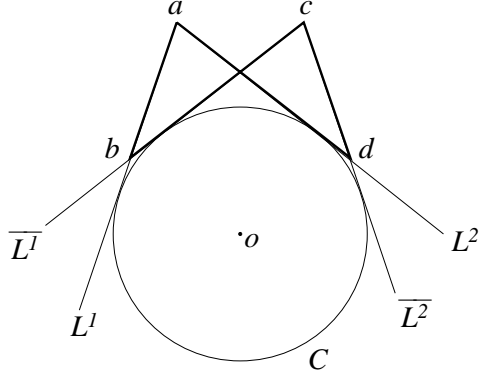}
\end{center}
\caption{Constructing a spherical antiparallelogram $abcd$}
\end{figure}

As a result, we get an antiparallelogram $abcd$ on $S$
that is circumscribed around the circle $C$.
We call the selfintersecting quadrilateral $abcd$ 
the antiparallelogram because its opposite sides
have equal lengths and we say that it is circumscribed 
around the circle $C$ because its sides lie on the lines that
are tangent to $C$.

\textit{Step II}: Let $C$ be the circle from Step I, 
$o\in S$ be the center of $C$, and $r_C$ be the radius of $C$. 
Consider a continuous family of circles
$C_r\subset S$ such that 
\begin{enumerate}[(i)]
\item $o$ is the center of $C_r$ for every $r$;
\item $C_r$ has radius $r$ for every $r$;
\item the circle $C$ belongs to this family.
\end{enumerate}

For every $r$, we repeat Step I.
More precisely,  we first select a continuous family of 
points $a_r$ such that the angle between the two (geodesic) 
lines $L^1_r$ and $L^2_r$ on $S$, that pass through $a_r$ 
and are tangent to $C_r$, is equal to the angle between the 
two (geodesic) lines on $S$ that pass through $a_{r_C}$ and 
are tangent to the circle~$C=C_{r_C}$. 
Then we  select such a rotation of the set $L^1_r\cup L^2_r$ 
around the center of $C_r$ that the angle between the lines
$L^1_r$  and $\overline{L^1_r}$ is equal to the angle between the lines
$L^1=L^1_{r_C}$ and $\overline{L^1}=\overline{L^1_{r_C}}$.
Here $\overline{L^j_r}$, $j=1,2$, stands for the image of the line
$L^j_r$ under the rotation. 
At last, let $c_r$ be the image of the point $a_r$ under 
the above rotation, $b_r=L^1_r\cap \overline{L^1_r}$, and 
$d_r=L^2_r\cap \overline{L^2_r}$. 

As a result, we obtain a continuous
family of antiparallelograms $a_rb_rc_rd_r$ on $S$
such that, for every $r$, $a_rb_rc_rd_r$ is  
circumscribed around the circle $C_r$
and has the same angles as~$abcd$.
The fact that the 
angles at the vertices $a_r$, $b_r$, and $c_r$ 
are equal to the corresponding angles of the
antiparallelogram  $abcd$ at the vertexes $a$, $b$, and $c$,
obviously, holds true by construction.
Due to the symmetry of
$a_rb_rc_rd_r$, the angle at the vertex $d_r$ is equal 
to the angle at the vertex $b_r$ and, thus, is equal to the
angle of the antiparallelogram  $abcd$ at the vertices $b$ and~$d$.

\textit{Step III}: Let  $a_rb_rc_rd_r$ be an antiparallelogram
constructed in Step II that lies on the sphere $S$,
is circumscribed around the circle $C_r$ of radius $r$,
depends continuously on $r$, and whose angles at the 
vertices $a_r$, $b_r$, $c_r$, and $d_r$ are 
independent of $r$.

Consider an infinite cone $K_r$ 
over the antiparallelogram $a_rb_rc_rd_r$ 
with apex at the center $O$ of the sphere $S$.
Draw a plane $\pi(s)$ in the 3-space that is perpendicular
to the line joining $O$ with the center $o$
of the circle $C_r$ and such that  $\pi(s)$
is at distance $s$ from $O$.
We claim that there is a continuous function $s(r)$ such that
the value  $\beta$ of the dihedral angle between the plane 
$\Pi_r$ of the triangle $a_rb_rO$ and the plane  $\pi(s(r))$  
is independent of $r$.

In fact, this statement immediately follows from 
the trigonometric relations for hyperbolic
right triangles. 
Consider the hyperbolic right triangle $Opq$, where
$p$ stands for the nearest point of the plane  $\pi(s(r))$
to the point $O$ and $q$ stands for the nearest point 
of the line  $\pi(s(r))\cap\Pi_r$ to the point $O$.
Then 
\begin{equation}\label{eq2}
\cos\beta = \cosh s(r) \sin r, 
\end{equation}
see, e.g., \cite[formula (3.5.16)]{Ra06}.
Using this equation, we find $s(r)$.

Since the antiparallelogram $a_rb_rc_rd_r$ 
is circumscribed around the circle $C_r$,
it follows that the dihedral angle between 
the plane $\pi(s(r))$ and each of the planes 
containing one of the triangles $b_rc_rO$, 
$c_rd_rO$, and $a_rd_rO$, is equal to $\beta$ and, 
in particular, is independent of~$r$.

\textit{Step IV}: Consider a continuous family of cones  $K_r$
constructed in Step III. 
Let $\sigma_r$ be the half-space determined by the plane 
$\pi(s(r))$ such that $O\in \sigma _r$.
By definition, put $K^+_r=K_r\cap\sigma_r$.
Let $K^-_r$ be obtained from $K^+_r$ by reflecting it
in the plane $\sigma_r$ and let $M_r=K^+_r\cup K^-_r$.

For every $r$, $M_r$ is a boundary-free polyhedral
surface with selfintersections that is combinatorially 
equivalent to the surface of the regular octahedron.

The reader familiar with the theory 
of flexible polyhedra \cite{Sa11}
can observe that the construction of $M_r$ has very much 
in common with the construction of the Bricard octahedra of
type~II. 

In the next Step we
will finalize the construction of a selfintersection-free
polyhedron, whose existence is proclaimed in Theorem \ref{t1},
using the trick that was originally proposed by
R. Connelly in the construction of his famous
selfintersection-free flexible polyhedron \cite{Co77}.

\textit{Step V}: Let $\widetilde{a}_r$ be the point of
the intersection of the line $a_rO$ and the plane $\pi(s(r))$.
Similarly we define the points $\widetilde{b}_r$,
$\widetilde{c}_r$, and $\widetilde{d}_r$.
Let $\widetilde{O}_r$ be symmetric to the point $O$
with respect to the plane $\pi(s(r))$.
Note that the points $O$, $\widetilde{a}_r$, $\widetilde{b}_r$,
$\widetilde{c}_r$, $\widetilde{d}_r$, and 
$\widetilde{O}_r$ are the vertices of the polyhedron $M_r$.

Observe that, if we remove the triangles 
$\widetilde{a}_r\widetilde{d}_rO$ and
$\widetilde{a}_r\widetilde{d}_r\widetilde{O}_r$
from the polyhedron $M_r$, we get a disk-homeomorphic
selfintersection-free polyhedral surface.
Denote it by $N_r$.

Recall that the dihedral angle between the triangles
$\widetilde{a}_r\widetilde{d}_rO$ and
$\widetilde{a}_r\widetilde{d}_r\widetilde{O}_r$
is equal to $2\beta$, where
$\beta$ is the dihedral angle between the plane 
$\Pi_r$ of the triangle $a_rb_rO$ and the plane  $\pi(s(r))$  
constructed in Step III.
In particular, this dihedral angle is independent of $r$.

Now, consider a tetrahedron $T=WXYZ$ such that $T$ is 
sufficiently large in comparison with the dimensions 
of the polyhedron $N_r$ and the dihedral angle of $T$, 
attached to the edge $XY$, is equal to $2\beta$.
Let us select the points $x_r, y_r, z_r$ and $w_r$ such that
\begin{enumerate}[(i)]
\item $x_r$ lies on the geodesic segment $XY$ 
sufficiently far from its end-points and
depends continuously on $r$;
\item $y_r$ lies on the geodesic segment $XY$,
depends continuously on $r$, and the distance between 
$x_r$ and $y_r$
is equal to the distance between $O$ and $\widetilde{O}$;
\item $z_r$ belongs to the triangle $XYZ$
and its distances from the points $x_r$ and $y_r$ are
equal to the distance between the points 
$\widetilde{a}_r$ and $O$
(and, thus between the points 
$\widetilde{a}_r$ and $\widetilde{O}_r$);
\item $w_r$ belongs to the triangle $XYW$
and its distances from the points $x_r$ and $y_r$ are
equal to the distance between the points 
$\widetilde{d}_r$ and $O$.
\end{enumerate}

From the fact that the dihedral angle of $T$, 
attached to the edge $XY$, is equal to $2\beta$ and
the conditions (ii)--(iv), it obviously follows that
there is an isometry $\varphi$  of the 
hyperbolic 3-space such that 
$\varphi(\widetilde{O}_r)=x_r$,
$\varphi(O)=y_r$,
$\varphi(\widetilde{d}_r)=z_r$, and
$\varphi(\widetilde{a}_r)=w_r$.

For every $r$,
let us remove the triangles $x_ry_rz_r$ and
$x_ry_rw_r$ from the polyhedral surface $T$
and replace the union of these triangles 
by the polyhedral surface $\varphi(N_r)$, see Figure~\ref{fig2}. 
Denote the resulting polyhedral surface by $P_r$.
We may also describe this transformation of $T$ into $P_r$
as follows: first, we produce a quadrilateral
hole on the polyhedral surface $T$ and, second, we glue
this hole with an isometric copy of the polyhedral
surface $N_r$.
\begin{figure}\label{fig2}
\begin{center}
\includegraphics[width=0.7\textwidth]{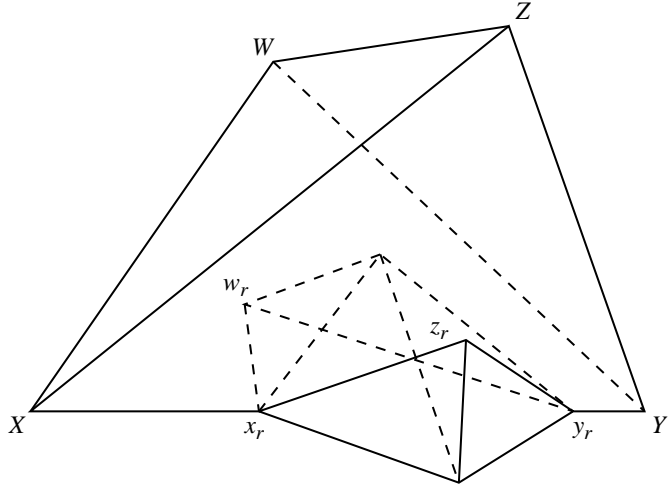}
\end{center}
\caption{Polyhedral surface $P_r$}
\end{figure}

Obviously, for some open interval $I\subset\mathbb R$, 
the family $\{P_r\}_{r\in I}$ is a continuous family
of nonconvex sphere-homeomorphic selfintersec\-tion-free
polyhedral surfaces such that every dihedral angle of $P_r$
is independent of~$r$. 

In the rest part of the proof we show that
all the statements of Theorem \ref{t1} are fulfilled for 
any polyhedron $P=P_r$, $r\in I$, i.e., that, as we vary $r$,  
the deformation of $P_r$ is nontrivial, the surface area 
and total mean curvature of $P_r$ as well as the Gauss 
curvature of some vertex of $P_r$ are nonconstant in~$r$.

\textit{Step VI}: In order to prove that the above constructed
deformation of the polyhedral surface $P_r$ is nontrivial,
it is sufficient to prove that the (spatial) distance 
between the points $x_r$ and $y_r$ is not constant in~$r$.

Observe that this distance is equal to $2s(r)$, where
the function $s(r)$ is defined in Step III as the distance from
the point $O$ to a plane. 
In particular, the function $s(r)$ satisfies the equation~\eqref{eq2}
and, obviously, is nonconstant on every interval of the reals.

Thus, the deformation of the polyhedral surface $P_r$ 
is nontrivial. 

\textit{Step VII}: Let's prove that the surface area 
of $P_r$ is not constant
for $r\in I$.

According to Step V, for $r\in I$,
$P_r$ is obtained from $T$ by
replacing the union of triangles $x_ry_rz_r$ and
$x_ry_rw_r$ with an isometric copy of the polyhedral
surface $N_r$. 
The latter consists of the triangle 
$\widetilde{O}_rO\widetilde{d}_r$ 
(which is isometric to the triangle $x_ry_rz_r$), 
the triangle $\widetilde{O}_rO\widetilde{a}_r$
(which is isometric to the triangle $x_ry_rw_r$), 
and four mutually isometric triangles
(each of which is isometric to the triangle 
$O\widetilde{a}_r\widetilde{b}_r$) whose 
surface area $S_r$ is positive.
Hence, the surface area of $P_r$ exceeds the 
surface area of $T$ by the strictly positive
number $4S_r$.

Using the notation of Step III, we can say that
the distance between the points $x_r$ and $y_r$
is equal to $2s(r)$ and satisfies the equation~\eqref{eq2},
namely, $\cos\beta = \cosh s(r) \sin r$,
where $2\beta$ stands for the dihedral
angle of $P_r$ at the edge $\widetilde{a}_r\widetilde{b}_r$.
Taking into account that $\beta$ is independent of $r$,
pass to the limit in the equation~\eqref{eq2} as
$r\to (\pi/2-\beta)-0$. As a result we get $s(r)\to 0$.
Moreover, if we consider the hyperbolic right triangle 
$Opq$ from Step III, we conclude that the distance
between every pair of the vertices of the polyhedral
surface $N_r$ tends to zero as $r\to (\pi/2-\beta)-0$.
Hence, for all values of the parameter $r$ that
are less than $\pi/2-\beta$ but sufficiently close
to $\pi/2-\beta$, the surface area of $P_r$ is
arbitrarily close to the surface area of $T$.

So, we see that 
the difference of the surface area of $P_r$ 
and the surface area of $T$ is strictly positive for every
$r\in I$ and tends to zero as $r\to (\pi/2-\beta)-0$.
Hence, the surface area of $P_r$ is nonconstant.
As far as this function is analytic in $r$,
it is nonconstant on every interval, in particular, on $I$.
Thus the surface area of $P_r$ is nonconstant
for $r\in I$.

\textit{Step VIII}: The proof of the fact that 
the total mean curvature 
of the polyhedral surface $P_r$ is nonconstant
for $r\in I$ is similar to Step VII.
More precisely, we observe that, for a given $r\in I$,
the total mean curvature of $P_r$ is strictly
greater than the total mean curvature of $T$ but
their difference tends to zero as $r\to (\pi/2-\beta)-0$
and, thus, is nonconstant.

\textit{Step IX}: Here we prove that the Gauss
curvature of the vertex $y_r\in P_r$ is nonconstant in $r$. 

Recall that the Gauss curvature of the vertex $y_r$
is equal to the difference between $2\pi$ and the sum of all 
plane angles of $P$ incident to $y_r$.

Using the notation introduced in Step V and the fact that
$P_r$ is obtained from $T$ by
replacing the union of the triangles $x_ry_rz_r$ and
$x_ry_rw_r$ with an isometric copy of the polyhedral
surface $N_r$, 
we observe that the Gauss curvature of $P_r$ at 
the vertex $y_r$ is equal to 
$-2\angle\widetilde{a}_rO\widetilde{b}_r=
-2a_rb_r$,
where $\angle\widetilde{a}_rO\widetilde{b}_r$ stands for the 
angle of the triangle $\widetilde{a}_rO\widetilde{b}_r$ attached 
to the vertex $O$ and $a_rb_r$ stands for the spherical
distance between the points $a_r$, $b_r$ on the unit
sphere.
This means that all we need is to prove that $a_rb_r$
is nonconstant in $r$.

Arguing by contradiction,
suppose we have two quadrilaterals 
$a_rb_rc_rd_r$ and $a_tb_tc_td_t$, 
constructed according to Step I, 
such that $a_rb_r=a_tb_t$. 
Since the corresponding angles of 
$a_rb_rc_rd_r$ and $a_tb_tc_td_t$ are equal to each other,
we conclude that these quadrilaterals are mutually congruent.
Hence, their circumradii are
equal to each other, i.e., $r=t$.
Thus, $a_rb_r\neq a_tb_t$ for $r\neq t$ and the Gauss
curvature of the vertex $y_r\in P_r$ is nonconstant in~$r$. 
This completes the proof of Theorem~\ref{t1}.
\end{proof}

\begin{thrm}\label{t2}
For every compact bounary-free 
oriented polyhedral surface $P$
in the hyperbolic  3-space and every smooth
deformation that leaves the dihedral angles of $P$ fixed,
the volume bounded by $P$ remains constant 
in the course of deformation.
\end{thrm}

\begin{proof}
Let $\{P_r\}_{r\in I}$ be a smooth
deformation of $P$ leaving the dihedral angles fixed
and let $\ell^j_r$, $j=1,\dots, J$, stand for
the edges of the polyhedron $P_r$.
If we denote the length of the edge $\ell^j_r$
by $|\ell^j_r|$ and the dihedral angle
of $P_r$ attached to $\ell^j_r$ by $\alpha^j_r$
then the classical Schl{\"a}fli differential formula \cite{So04}
reads as follows
\begin{equation}\label{eq3}
\frac{d}{dr}\mbox{vol\,}P_r= 
-\frac12 \sum_j \bigl|\ell^j_r\bigr| \frac{d}{dr}\alpha^j_r.
\end{equation} 
This completes the proof of Theorem \ref{t2},
since,  by assumption, $\frac{d}{dr}\alpha^j_r=0$
for all $j=1,\dots, J$.
\end{proof}

\begin{remark}
The statement and the proof of Theorem~\ref{t2}
hold true both for polyhedra with or without 
selfintersections. For more details, including 
a formal definition of a polyhedron with selfintersections,
the reader is
referred to the theory of flexible polyhedra~\cite{Sa11}. 
\end{remark}

\begin{thrm}\label{t3}
In the hyperbolic 3-space $\mathbb{H}^3$,
there exists a tiling composed of congruent polyhedral tiles
that possesses a nontrivial continuous deformation
such that, in the course of deformation,  
the dihedral angles of every tile are left fixed, 
and the union of the deformed tiles produces a tiling 
composed of congruent polyhedral tiles again.
\end{thrm}

\begin{proof}
Our proof makes use of the so-called 
B{\"o}r{\"o}czky tiling
of $\mathbb{H}^3$ by congruent polyhedra and,
for the reader's convenience, we start with a short 
description of this tiling.
Figure~\ref{fig3} illustrates the construction of 
the B{\"o}r{\"o}czky tiling
in the upper half-space model.
\begin{figure}\label{fig3}
\begin{center}
\includegraphics[width=0.7\textwidth]{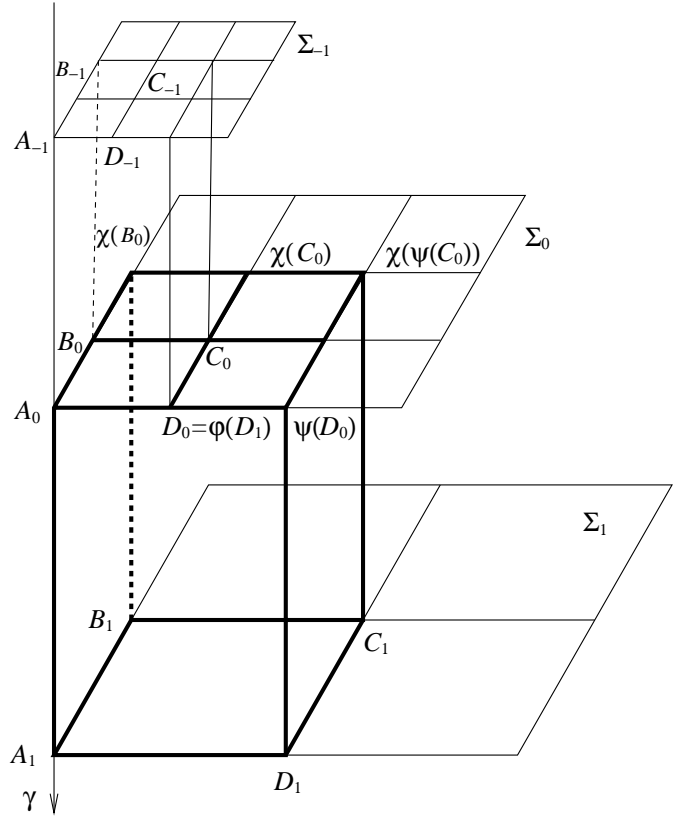}
\end{center}
\caption{Constructing the B{\"o}r{\"o}czky tiling 
in the upper half-space model of the hyperbolic  3-space.
The tile $\varkappa$ is shown by bold lines}
\end{figure}

Suppose the hyperbolic  3-space has curvature $-1$ 
and let $\Sigma_0$ be an horosphere.
It is well-known that $\Sigma_0$ is isometric to 
the Euclidean plane.
Let $\pi_0$ be an edge-to-edge tiling of $\Sigma_0$
with pairwise equal geodesic squares with edge length 1.
Fix one of the squares of $\pi_0$ and 
denote its vertices by $A_0$, $B_0$, $C_0$, and $D_0$ 
as shown in Figure~\ref{fig3}. 
Let $\gamma$ be an oriented line through $A_0$
that is orthogonal to $\Sigma_0$.
Starting from $A_0$, place points $A_k$,  $k\in\mathbb Z$,
on $\gamma$ such that $A_k$ precedes $A_{k+1}$ on 
the oriented line $\gamma$
and the hyperbolic distance
between $A_k$ and $A_{k+1}$ is equal to
$\ln 2$ for all $k\in\mathbb Z$.
Through every point $A_k$ draw an horocycle $\Sigma_k$ 
orthogonal to~ $\gamma$.

Let $\varphi$ be an orientation preserving 
isometric mapping of $\mathbb{H}^3$ onto itself 
that maps the point $A_0$ into the point $A_{-1}$,
maps the line $\gamma$ onto itself and 
maps the plane through $\gamma$ and $B_0$ onto itself.
Starting from the tiling $\pi_0$ of the horosphere $\Sigma_0$, 
define the tiling $\pi_k$ of the horosphere $\Sigma_k$
by putting $\pi_k=\varphi(\pi_{k+1})$ for all $k\in\mathbb Z$.
Note that  $A_{k}=\varphi(A_{k+1})$  for all $k\in\mathbb Z$.
By definition, put $B_{k}=\varphi(B_{k+1})$, 
$C_{k}=\varphi(C_{k+1})$, and $D_{k}=\varphi(D_{k+1})$. 

Let $\psi$ (respectively, $\chi$) be an orientation preserving 
isometry of $\mathbb{H}^3$ onto itself 
that maps the square $A_0B_0C_0D_0$ onto its neighbour 
in the tiling $\pi_0$ in such a way that
$\psi(A_0)=D_0$ and $\psi(B_0)=C_0$ 
(respectively, $\chi(A_0)=B_0$ and $\chi(D_0)=C_0$).

A cell of the B{\"o}r{\"o}czky tiling
is a (nonconvex solid) polyhedron with the following
13 vertices: 
$A_1$, $B_1$, $C_1$, $D_1$, 
$A_0$, $B_0$, $C_0$, $D_0$, 
$\psi(C_0)$, $\psi(D_0)$, $\chi(B_0)$, 
$\chi(C_0)$, and $\chi(\psi(C_0))=\psi(\chi(C_0))$.
Its combinatorial structure is shown in Figure~\ref{fig3} 
with bold lines.
Denote this polyhedron by $\varkappa$.

The main observation, allowing to build 
the  B{\"o}r{\"o}czky tiling is that, 
since the hyperbolic distance between the points $A_0$ and 
$A_1$ is equal to $\ln 2$ and the curvature of the space 
is equal to $-1$, 
the Euclidean distance between the points 
$A_0$ and $\psi(D_0)$ (measured in the Euclidean plane $\Sigma_0$) 
is twice the Euclidean distance between the points $A_0$ 
and $D_0$ and, thus, twice the distance  between the points 
$A_1$ and $D_1$.
Similarly, the Euclidean distance between the points 
$A_0$  and $\chi(B_0)$ 
is twice the Euclidean distance between 
the points $A_0$ and $B_0$.
If we observe now that, for every $k\in\mathbb Z$,
$\varphi$ maps the tiling $\pi_k$ of the horosphere $\Sigma_k$
onto the tiling $\pi_{k-1}$ of the horosphere $\Sigma_{k-1}$ and
both $\psi$ and $\chi$ map
the tiling $\pi_k$ of the horosphere $\Sigma_k$
onto itself, we can complete the description of the 
 B{\"o}r{\"o}czky tiling as follows.

We apply iterations $\varphi^k$, $k\in\mathbb Z$,
of the hyperbolic isometry $\varphi$ to the polyhedron $\varkappa$
and get a sequence of polyhedra $\varphi^k(\varkappa)$
whose 9 `upper' vertices lie on the horosphere $\Sigma_k$ 
(and belong to the set of the vertices of the tiling $\pi_k$) 
and 4 `bottom' vertices lie on the horosphere $\Sigma_{k+1}$
(and belong to the set of the vertices of the tiling $\pi_{k+1}$).
Then we fix  $k\in\mathbb Z$ and apply iterations 
$\psi^p$, $p\in\mathbb Z$, 
and $\chi^q$, $q\in\mathbb Z$,  
of the isometries $\psi$ and $\chi$ to the polyhedron 
$\varphi^k(\varkappa)$. 
As a result we get a tiling of a
`polyhedral layer' with vertices on $\Sigma_k$ and $\Sigma_{k+1}$. 
These `polyhedral layers' fit together and produce 
the  B{\"o}r{\"o}czky tiling of the whole 3-space.

In short, we can say that the B{\"o}r{\"o}czky tiling 
is produced from the cell $\varkappa$ by isometries 
$\varphi$, $\psi$, and $\chi$.

For more details about the B{\"o}r{\"o}czky tiling
the reader may consult \cite{DF10} and references 
given there.

We now turn to the proof of Theorem 3 directly.
Let's construct a polyhedral surface $P\subset\mathbb S^3$ 
that is described in the statement of Theorem 1
such that all dimensions of $P$ are sufficiently
small in comparison with the size of the polyhedron~$\varkappa$. 
For the vertices of $P$, we 
use the notations from Step V of the proof
of Theorem \ref{t1} (or, identically, from Figure~\ref{fig2}). 

On the face $A_1A_0B_0\chi(B_0)B_1$ of the boundary
$\partial\varkappa$ of the solid polyhedron $\varkappa$,
find a (hyperbolic) triangle
$\Delta$ that is congruent to the triangle $WYZ$
of the polyhedral surface $P$ (see Figure~\ref{fig2}) and
lies sufficiently far from all the vertices of~$\varkappa$.
Remove  $\Delta$ from the polyhedral
surface $\partial\varkappa$ and glue the hole obtained by an
isometric copy $\Sigma$ of the disk-homeomorphic 
polyhedral surface remaining after the removal 
of $WYZ$ from the polyhedral surface~$P$.
From the resulting sphere-homeomorphic surface remove
a triangle $\psi(\Delta)$ and 
glue the hole obtained by 
the disk-homeomorphic surface $\psi(\Sigma)$.
Denote the resulting 
sphere-homeomorphic polyhedral surface 
by~$\partial\overline{\varkappa}$.

The compact part $\overline{\varkappa}$
of the hyperbolic  3-space bounded by the 
sphere-homeo\-morphic polyhedral surface 
$\partial\overline{\varkappa}$ is the cell of the 
tiling whose existence is proclaimed in Theorem~\ref{t3}.
The tiling itself
is produced from the cell $\overline{\varkappa}$ by isometries 
$\varphi$, $\psi$, and $\chi$ precisely in the same way
as these isometries produce the B{\"o}r{\"o}czky tiling 
from its cell $\varkappa$.
The nontrivial continuous deformation that
preserves the dihedral angles of the cell
is produced by the corresponding deformation of $P$.
\end{proof}

\section{Polyhedra in the spherical 3-space}

In the spherical 3-space $\mathbb S^3$ we may prove 
the same statements as in the hyperbolic  3-space.
For the reader's convenience, below we formulate 
Theorems \ref{t4}--\ref{t6} that hold true in  $\mathbb S^3$
and are similar to 
Theorems \ref{t1}--\ref{t3} and give brief 
comments on their proofs.

\begin{thrm}\label{t4}
In the spherical 3-space,
there exists a nonconvex sphere-homeomor\-phic 
polyhedron $P$ with the following properties:
\begin{enumerate}[(i)]
\item $P$ has no selfintersections;
\item $P$ admits a continuous family of 
nontrivial deformations which leave the dihedral 
angles fixed;
\item the surface area of $P$ is nonconstant
during the deformation;
\item the total mean curvature  of $P$ is nonconstant
during the deformation;
\item the Gauss curvature of some vertex of $P$
is nonconstant.
\end{enumerate}
\end{thrm}

\begin{proof}
Up to an obvious replacement of theorems of hyperbolic trigonometry
by the corresponding theorems of spherical trigonometry,
the proof is similar to the proof of Theorem \ref{t1}
\end{proof}

\begin{thrm}\label{t5}
For every compact bounary-free 
oriented polyhedral surface $P$
in the spherical 3-space and every smooth
deformation that leaves the dihedral angles of $P$ fixed,
the volume bounded by $P$ remains constant 
in the course of deformation.
\end{thrm}

\begin{proof}
The result follows directly from the 
classical Schl{\"a}fli differential formula 
for the spherical 3-space \cite{So04},
which may be obtained from the formula~\eqref{eq3}
if we multiplay the right-hand side by~$-1$. 
\end{proof}

\begin{remark}
Recall an open problem: 
\textit{Prove that if all dihedral angles of a simplex in 
the spherical 3-space $\mathbb S^3$ are rational 
multiples of~$\pi$ then the 
volume of this simplex is a rational multiple of $\pi^2$},
see \cite{DS00}, \cite{Du12}.\footnote{In Section 18.3.8.6 of 
the well known book \cite{Be09} the reader may find a
hypothesis that this problem should be solved in a negative form.
We just want to warn the reader about a typo
which may obscure: the last letter $\pi$
in that Section 18.3.8.6 should be replaces by $\pi^2$.} 
If one be interested in the study of a similar problem 
not only for simplices but for all polyhedra,
he may be tempted to construct a polyhedron in $\mathbb S^3$, 
with all dihedral angles being rational multiples of $\pi$, 
admitting continuous deformations that preserve its 
dihedral angles but change its volume. 
In this case, the volume, being a continuous function, 
takes every value in some interval of real numbers and,
among others, takes values that are not rational multiples of $\pi^2$.
Hence, this argument, if it is correct, will easily result
in a negative solution to the above problem extended to all polyhedra.
Nevertheless, Theorem \ref{t5} shows that this argument is not 
applicable because the volume is necessarily constant.
\end{remark}

\begin{thrm}\label{t6}
In the spherical 3-space $\mathbb S^3$,
there exists a tiling composed of congruent polyhedral tiles
that possesses a nontrivial continuous deformation
such that, in the course of deformation,  
the dihedral angles of every tile are left fixed, 
and the union of the deformed tiles produces a tiling 
composed of congruent polyhedral tiles again.
\end{thrm}

\begin{proof}
Our arguments are similar to the proof of Theorem \ref{t3}.
Moreover, they are even simpler 
because now we can use a finite partition of~$\mathbb S^3$ 
instead of the B{\"o}r{\"o}czky tiling of 
the hyperbolic  3-space. 

For example, let's use the tiling
of~$\mathbb S^3$ with 12 mutually congruent polyhedra
obtained in the following way.
Let's treat~$\mathbb S^3$ as the standard unit sphere in
$\mathbb R^4$ centered at the origin.
Let $N=(0,0,0,1)$, $S=(0,0,0,-1)$ and $L$ be the subspace 
of $\mathbb R^4$ orthogonal to the vector $(0,0,0,1)$.
By definition, put $\mathcal S=\mathbb S^3\cap L$. 
Then $\mathcal S$ is a unit 2-sphere located in the 3-space $L$.
Let $C\subset L$ be a 3-dimensional cube inscribed in~$\mathcal S$.
Then the images of the 2-faces of $C$ 
under the central projection from the origin 
$(0,0,0,0)$ into the 2-sphere $\mathcal S$ form a tiling
of~$\mathcal S$ with 6 mutually congruent 
spherical convex polygons that we denote by $T_j$, $j=1,\dots, 6$.
Joining each vertex of $T_j$ with $N$ 
we get six mutually congruent spherical convex polyhedra
$T^N_j$, $j=1,\dots,6$.
Similarly, joining each vertex of $T_j$ with $S$ 
we get six mutually congruent spherical convex polyhedra
$T^S_j$, $j=1,\dots,6$.
Obviously, 12 mutually congruent polyhedra $T^N_j$, 
$T^S_j$, $j=1,\dots,6$, tile the spherical 3-space $\mathbb S^3$. 

Observe that, for every $j=1,\dots,6$, the two polyhedra 
$T^N_j$ and $T^S_j$ are centrally symmetric to each other
with respect to the center of symmetry $O_j$ of the
polygon~$T_j$.

Let's construct a polyhedron $P\subset\mathbb S^3$ 
that is described in Theorem \ref{t4}
such that all dimensions of $P$ are sufficiently
small in comparison with the size of~$T^N_1$. 
For the vertices of $P$, we 
use the same notations as at Step V of the proof
of Theorem \ref{t1} (or, equivalently, as on Figure~\ref{fig2}). 

On the polygon $T_1$, find a (spherical) triangle
$\Delta$ that is congruent to the triangle $WYZ$
of the polyhedral surface $P$ (see Figure~\ref{fig2}) and
lies sufficiently far from all the vertices of
$T^N_1$ and from the point $O_1$.
Remove the triangle $\Delta$ from~ $T^N_1$ 
and glue the hole obtained by an
isometric copy $\Sigma$ of the disk-homeomorphic 
polyhedral surface remaining after the removal 
of~ $WYZ$ from~ $P$.
From the resulting sphere-homeomorphic surface remove
a triangle $\Delta'$ that is symmetric to~$\Delta$ 
with respect to the point $O_1$ and 
glue the hole obtained by a
disk-homeomorphic surface that is symmetric to~$\Sigma$
with respect to~$O_1$.
Denote the resulting sphere-homeomorphic surface by
$\overline{T^N_1}$.

Denote by $\overline{T^S_1}$ the sphere-homeomorphic 
surface that is symmetric to $\overline{T^N_1}$ with 
respect to~$O_1$.

Let $\varphi_1:\mathbb S^3\to\mathbb S^3$ be the identity
mapping and, for every $j=2,\dots,6$, 
$\varphi_j:\mathbb S^3\to\mathbb S^3$ be an isometry
such that $\varphi_j(T_1)=T_j$,
$\varphi_j(N)=N$ and $\varphi_j(S)=S$.

Consider the union of 12 sphere-homeomorphic surfaces
$\varphi_j(\overline{T^N_1})$ and
$\varphi_j(\overline{T^S_1})$, $j=1,\dots,6$.
Obviously, these surfaces define a tiling of $\mathbb S^3$
by pairwise congruent polyhedral tiles 
that possess a nontrivial continuous deformation
(which is produced by the corresponding deformation of $P$)
such that, in the course of deformation,  
the dihedral angles of every tile are left fixed, 
and the union of the deformed tiles produces a tiling 
composed of congruent tiles again.
This concludes the proof of Theorem~\ref{t6}.
\end{proof}

\end{document}